\documentclass[10pt,reqno,a4paper]{amsart}
\usepackage{graphicx}
\usepackage[nocompress]{cite}
\usepackage{tikz-cd}
\usepackage{subcaption}
\usepackage{mathrsfs}
\usepackage{amsopn,amssymb,amsmath}
\usepackage{url}
\usepackage{subcaption}
\usepackage{mleftright}
\usepackage{verbatim}
\usepackage[hidelinks]{hyperref}

\usetikzlibrary{calc,patterns,angles,quotes}
\theoremstyle{definition}
\newtheorem{theorem}{Theorem}[section]
\newtheorem{prop}[theorem]{Proposition}
\newtheorem{defn}[theorem]{Definition}
\newtheorem{cor}[theorem]{Corollary}
\newtheorem{eg}[theorem]{Example}
\newtheorem{lemma}[theorem]{Lemma}
\newtheorem{remark}[theorem]{Remark}
\numberwithin{equation}{section}

\DeclareMathOperator{\Ima}{Im}
\DeclareMathOperator{\rank}{rank}

\DeclareMathOperator{\Hom}{Hom}

\DeclareMathOperator{\Harm}{Harm}

\DeclareMathOperator{\id}{id}
\DeclareMathOperator{\diag}{diag}
\DeclareMathOperator{\Ann}{Ann}
\newcommand{\Z}{\mathbb{Z}}
\newcommand{\R}{\mathbb{R}}
\newcommand{\C}{\mathbb{C}}
\newcommand{\F}{\mathbb{F}}

\begin{document}

\title{Weighted (co)homology and weighted Laplacian}

\author{Chengyuan Wu}
\address{Department of Mathematics, National University of Singapore, Singapore 119076}
\email{wuchengyuan@u.nus.edu}
\thanks{Chengyuan Wu, Shiquan Ren, Jie Wu, Kelin Xia -- First authors.}
\thanks{The project was supported in part by the Singapore Ministry of Education research grant (AcRF Tier 1 WBS No.~R-146-000-222-112). The first author was supported in part by the President's Graduate Fellowship of National University of Singapore. The second author was supported by China Postdoctoral Science Foundation (grant no.\ 2022M721023). The third author was supported by the Natural Science Foundation of China (NSFC grant no.\ 11971144) and High-level Scientific Research Foundation of Hebei Province. The fourth author was supported by Singapore Ministry of Education Academic Research fund MOE-T2EP20120-0013 and MOE-T2EP20220-0010.}

\author{Shiquan Ren}
\address{School of Mathematics and Statistics, Henan University, Kaifeng 475004, China}
\email{renshiquan@henu.edu.cn}


\author{Jie Wu}
\address{Yanqi Lake Beijing Institute of Mathematical Sciences and Applications, Beijing 101408, China}
\email{wujie@bimsa.cn}

\author{Kelin Xia}
\address{School of Physical and Mathematical Sciences, Nanyang Technological University, Singapore 637371}
\email{xiakelin@ntu.edu.sg}





\begin{abstract}
In this paper, we generalize the combinatorial Laplace operator of Horak and Jost by introducing the $\phi$-weighted coboundary operator induced by a weight function $\phi$. Our weight function $\phi$ is a generalization of Dawson's weighted boundary map. We show that our above-mentioned generalizations include new cases that are not covered by previous literature. Our definition of weighted Laplacian for weighted simplicial complexes is also applicable to weighted/unweighted graphs and digraphs.
\end{abstract}

\maketitle
\section{Introduction}
The Laplacian is an important subject of study, with applications in networked control and sensing \cite{muhammad2006control} and statistical ranking \cite{jiang2011statistical}. In the recent seminal paper by Horak and Jost \cite{Horak2013}, important progress on the Laplacian was made. The combinatorial Laplacian $\mathcal{L}(K,w_K)$ in \cite{Horak2013} unifies all Laplace operators studied to date \cite{duval2002shifted,friedman1998computing,grone1990laplacian,chung1996combinatorial,banerjee2008spectrum}. In this paper, we study the weighted Laplacian resulting from a weighted coboundary operator, induced by a weight function $\phi$. By analyzing the methodology in \cite{Horak2013}, we apply the $\phi$-weighted coboundary operator to generalize \cite{Horak2013}.
In Section \ref{subsec:matweightlap}, we show that our $\phi$-weighted Laplacian includes new cases not covered in Horak and Jost \cite{Horak2013}.

In 1990, Robert J.\ Dawson \cite{Dawson1990} was the first to study in depth the homology of weighted simplicial complexes. He studied the weighted boundary map as well as the categorical properties of weighted simplicial complexes. S.~Ren, C.~Wu and J.~Wu \cite{ren2018weighted,ren2017further} then explored further generalizations of weighted simplicial complexes and applications to (weighted) persistent homology \cite{bell2017weighted,edelsbrunner2012persistent,Zomorodian2005,wu2020weighted}. 

In this paper, we study the weighted (co)homology of simplicial complexes by considering weighted (co)boundary operators. In principle, the results in this paper can be generalized to any cellular space, though we focus on the case of simplicial complexes. Our weight function is a generalization of Dawson's weighted boundary map \cite{Dawson1990}, and includes new cases not covered in \cite{Dawson1990}, which will be described in Section \ref{subsec:dawson}. Our weight function can also be viewed as a generalization of the S-complex, which was introduced in \cite{mrozek2009coreduction} by Mrozek and Batko.

In general, one of the motivations of weighted (co)homology is to distinguish different elements in a data set. In the papers \cite{cang2017topologynet,cang2018persistent}, Z.\ Cang and G.\ Wei have used element-specific persistent homology and cohomology for the analysis of proteins to excellent effect. In the paper \cite{ren2018weighted}, the authors show that weighted persistent homology can tell apart filtrations that ordinary persistent homology does not distinguish. Our paper can be viewed as a further exploration of \cite{ren2018weighted} by studying various kinds of weight functions.

For applications, we show in Section \ref{subsec:biomolecular} that $\phi$-weighted homology can differentiate between different weighted polygons. A motivating example is the case of ring structures in biomolecules, by setting bond angles as weights. We derive and prove a formula for the $\phi$-weighted homology of a pentagonal ring structure, and generalize it for a $n$-gon in Theorem \ref{thm:ngon}. 

Our next application is regarding the $\phi$-weighted Laplacian of digraphs. The Laplacian of digraphs have been studied by Y.\ Li and Z.\ Zhang in \cite{li2010random,li2012digraph}. We develop a different way of defining the Laplacian of digraphs. The motivating example is the topic of network motifs (which are recurrent sub-graphs or patterns), which has found important applications in biology \cite{alon2007network,ferrell2013feedback,milo2002network,masoudi2012building}. In Section \ref{subsec:network}, we demonstrate how the eigenvalues and eigenvectors of the $\phi$-weighted Laplacian can distinguish between the 8 different types of feedforward loops (FFLs), which are certain types of digraphs. An advantage of our approach for the weighted Laplacian for digraphs is that it generalizes easily to higher dimensions (directed simplicial complexes), and also to cases where there are more than two types of arrows (e.g.\ weighted graphs/digraphs). 

\subsection{Other Related Work}
In 2007, Davis, Dymara, Januszkiewicz and Okun \cite{davis2007weighted} studied the weighted $L^2$-cohomology of Coxeter groups, using another approach specialized to a Coxeter system. In \cite{li2015twisted}, J.~Y.~Li, V.~V.~Vershinin and J.~Wu study twisted simplicial groups and twisted homology of categories. 

In \cite{hansen2018toward}, J.\ Hansen and R.\ Ghrist outlines a program of spectral sheaf theory, which is an extension of spectral graph theory to cellular sheaves. They study weighted cellular sheaves, which is a cellular sheaf with values in $\F$-vector spaces.

In \cite{sharma2017weighted}, A.\ Sharma, T.J.\ Moore, A.\ Swami and J.\ Srivastava study the evolution and formation of small collaborations effectively using a hybrid approach of weighted simplicial complexes. Their work is effective in overcoming the shortcomings of traditional group network modeling approaches.

In addition, the authors O.T.\ Courtney and G.\ Bianconi \cite{courtney2017weighted} study a nonequilibrium model for weighted growing simplicial complexes. Their proposed method is able to generate weighted simplicial complexes with a rich interplay between weights and topology at various levels.

In the paper \cite{giusti2016two}, the authors C.\ Giusti, R.\ Ghrist and D.S.\ Bassett explore the use of simplicial complexes to study problems in neuroscience. The initial results are very promising and represent an important step towards understanding higher-order structure in neural data.

\section{$\phi$-Weighted Boundary Operator}
Given a simplicial complex $K$ and an abelian group $G$ as coefficient group, the $n$-th chain group $C_n(K;G)$ is a direct sum of copies of $G$ labelled by the $n$-simplices of $K$. Let $G$ be an $R$-module for some commutative ring with unity $R$. (Such an $R$ always exists since an abelian group is also a $\Z$-module.) We consider weights $\phi(\sigma,d_i\sigma)$ on $n$-simplices and its $i$-th face as elements of $R$. In this paper, all rings are assumed to be commutative rings with unity.

Throughout this paper, we use additive notation for the abelian group $G$. We fix a total order of the vertices of $K$, and impose the convention that a positively oriented simplex $\sigma\in K$ is always written with vertices in ascending order, that is, $\sigma=[v_{a_0},v_{a_1},\dots,v_{a_n}]$ where $a_0< a_1<\dots< a_n$. For a positively oriented simplex $\sigma=[v_{a_0},v_{a_1},\dots,v_{a_n}]$, we define the face map $d_i$ ($i\geq 0$) by $d_i\sigma=[v_{a_0},\dots,\widehat{v_{a_i}},\dots,v_{a_n}]$ (deleting the $i$-th vertex) so that $d_i\sigma$ is well-defined for each $i$. For a positively oriented simplex $\sigma$, we denote by $-\sigma$ the same simplex with negative orientation, and define $d_i(-\sigma)=-d_i(\sigma)$.

\begin{defn}[Weight function]
\label{defn:weightfn}
Let $G$ be an $R$-module.
We define a \emph{weight function} to be a function $\phi: K\times K\to R$, such that $\phi$ satisfies 
\begin{equation}
\label{eqn:weightcond}
\phi(d_i\sigma,d_jd_i\sigma)\phi(\sigma,d_i\sigma)g=\phi(d_j\sigma,d_jd_i\sigma)\phi(\sigma,d_j\sigma)g
\end{equation}
for all $\sigma\in K$, $g\in G$ and $j<i$.
\end{defn}


\begin{remark}
\label{remark:notface}
We notice that (\ref{eqn:weightcond}) does not depend on the values $\phi(\sigma,\tau)$ where $\tau$ is not a face of $\sigma$. Hence for convenience, we can define $\phi(\sigma,\tau)=0$ if $\tau$ is not a face of $\sigma$. If $\sigma$ is a 0-simplex, then $d_i\sigma$ only makes sense if $i=0$, hence there is no $j\geq 0$ that satisfies $j<i$. Hence, we can also define $\phi(\sigma,d_i\sigma)=0$ if $\sigma$ is a 0-simplex, such that \eqref{eqn:weightcond} is satisfied. Thus, in general when defining the weight function we can omit mentioning values of $\phi(\sigma,\tau)$ where either $\tau$ is not a face of $\sigma$, or $\sigma$ is a 0-simplex.
\end{remark}

\begin{remark}
It may seem on the surface that Definition \ref{defn:weightfn} is too strong and hence restricts the applicability of the theory. However, that is not the case, and in fact the usual boundary map (Example \ref{eg:id}) and the weighted boundary map in \cite{Dawson1990} (Section \ref{subsec:dawson}) both arise from weight functions that satisfy Definition \ref{defn:weightfn}.
\end{remark}

\begin{defn}
\label{defn:weightbm}
Given an $n$-chain $\alpha=\sum_{\sigma} g_\sigma \sigma$, where $\sigma$ runs over the $n$-simplices of $K$, we define the \emph{$\phi$-weighted boundary map} $\partial_n: C_n(K;G)\to C_{n-1}(K;G)$ by 
\begin{equation*}
\partial_n(g\sigma)=\sum_i (-1)^i \phi(\sigma,d_i\sigma)g d_i\sigma,
\end{equation*}
where $\phi$ is a weight function. We extend $\partial_n$ linearly to $C_n(K;G)$ by defining $\partial_n(\sum_{\sigma} g_\sigma \sigma):=\sum_{\sigma} \partial_n(g_\sigma\sigma)$.
\end{defn}

\begin{remark}
Unless otherwise specified, for ease of notation the symbol $\partial$ will refer to the \emph{weighted} boundary map in this paper.
\end{remark}

The following proposition justifies our definition of the weight function. As proved in the next proposition, the condition in Definition \ref{defn:weightfn} is equivalent to the condition $\partial_{n-1}\partial_n=0$.

\begin{prop} Let $\phi: K\times K\to R$ be a function.
Let $\partial_n: C_n(K;G)\to C_{n-1}(K;G)$ be defined as in Definition \ref{defn:weightbm}. Then $\partial_{n-1}\partial_n=0$ holds if and only if $\phi$ is a weight function.
\end{prop}
\begin{proof}
We calculate that
\begin{equation}
\label{eq:partial2working}
\begin{split}
\partial_{n-1}\partial_n(g\sigma)&=\partial_{n-1}\left(\sum_i(-1)^i\phi(\sigma,d_i\sigma)gd_i\sigma\right)\\
&=\sum_i(-1)^i\phi(\sigma,d_i\sigma)g\partial_{n-1}d_i\sigma\\
&=\sum_{j<i}(-1)^j \phi(d_i\sigma,d_jd_i\sigma)(-1)^i \phi(\sigma,d_i\sigma)gd_jd_i\sigma\\
&\quad+\sum_{j>i}(-1)^{j-1}\phi(d_i\sigma,d_{j-1}d_i\sigma)(-1)^i\phi(\sigma,d_i\sigma)gd_{j-1}d_i\sigma\\
&=\sum_{j<i}(-1)^{i+j}\phi(d_i\sigma,d_jd_i\sigma)\phi(\sigma,d_i\sigma)gd_jd_i\sigma\\
&\quad+\sum_{j>i}(-1)^{i+j-1}\phi(d_i\sigma,d_{j-1}d_i\sigma)\phi(\sigma,d_i\sigma)gd_{j-1}d_i\sigma\\
&=\sum_{j<i}(-1)^{i+j}\phi(d_i\sigma,d_jd_i\sigma)\phi(\sigma,d_i\sigma)gd_jd_i\sigma\\
&\quad+\sum_{i>j}(-1)^{i+j-1}\phi(d_j\sigma,d_{i-1}d_j\sigma)\phi(\sigma,d_j\sigma)gd_{i-1}d_j\sigma.
\end{split}
\end{equation}

Suppose $\partial_{n-1}\partial_n=0$. Then with the help of the identity $d_jd_i=d_{i-1}d_j$ for $j<i$, (\ref{eq:partial2working}) implies Equation \ref{eqn:weightcond} holds for all $\sigma\in K$ and $j<i$. Hence $\phi$ is a weight function.

Conversely, suppose $\phi$ is a weight function, i.e.\ Equation \ref{eqn:weightcond} holds for all $\sigma\in K$, $g\in G$ and $j<i$. Using the same calculations above, we can verify that $\partial_{n-1}\partial_n=0$.
\end{proof}

\begin{eg}
\label{eg:id}
Let $\phi(\sigma,d_i\sigma)\equiv 1_R$ for all $\sigma\in K$ and all $i$. We can verify that $\phi$ is a weight function. We call $\phi$ the \emph{identity weight function}. We observe that the $\phi$-weighted boundary map reduces to the usual boundary map when $\phi$ is the identity weight function.
\end{eg}

\begin{eg}
\label{eg:zero}
Let $\phi(\sigma,d_i\sigma)\in\Ann_R(G)$ for all $\sigma\in K$ and all $i$, i.e.\ $\phi(\sigma,d_i\sigma)g=0$ for all $g\in G$. We can verify that $\phi$ is a weight function, which we call the \emph{zero weight function}.
\end{eg}

\section{The case $G=\Z$, $R=\Z$}
We first consider the case where $G=\Z$, $R=\Z$. That is, we consider $G=\Z$ as a $\Z$-module (abelian group).


The condition for being a weight function (\ref{eqn:weightcond}) has a simplified version when $G=\mathbb{Z}$ and $R=\Z$.

\begin{prop}
\label{prop:integercond}
A function $\phi: K\times K\to\Z$ is a weight function if and only if
\begin{equation}
\label{eq:integercond}
\phi(\sigma,d_i\sigma)\cdot\phi(d_i\sigma,d_jd_i\sigma)=\phi(\sigma,d_j\sigma)\cdot\phi(d_j\sigma,d_jd_i\sigma)
\end{equation}
for all $\sigma\in K$ and for all $j<i$.
\end{prop}
\begin{proof}
We observe that $\Z$ is a faithful $\Z$-module, i.e. $\Ann_\Z(\Z)=0$. This can be seen from the fact that $1$ is annihilated only by 0. Let $r,s\in\Z$. Then, $r\cdot g=s\cdot g$ for all $g\in\Z$ if and only if $r=s$. This implies that (\ref{eqn:weightcond}) holds if and only if (\ref{eq:integercond}) holds.
\end{proof}

\subsection{Classification of Weight Functions}
\label{subsec:class}
As mentioned in Examples \ref{eg:id} and \ref{eg:zero}, we have the identity weight function $\phi(\sigma,d_i\sigma)\equiv 1$ and the zero weight function $\phi(\sigma,d_i\sigma)\equiv 0$, for all $\sigma\in K$ and all $i$.

We now proceed to study some nontrivial weight functions. In particular, we show that the definition of the weighted boundary map in \cite{Dawson1990} is a special case of our Definition \ref{defn:weightbm}.

\subsubsection{Semi-trivial Weight Functions}
Let $A,B\subseteq K$ be subsets (not necessarily subcomplexes) such that $A\cup B=K$. We construct the following class of semi-trivial weight functions.
\begin{prop}
Let $a:K\times K\to\Z$ be any arbitrary function. Consider the function $\phi: K\times K\to\Z$ such that
\begin{equation*}
\phi(\sigma,d_i\sigma)=\begin{cases}
0 &\text{if $\sigma\in A$}\\
0 &\text{if $d_i\sigma\in B$}\\
a(\sigma,d_i\sigma) &\text{otherwise}.
\end{cases}
\end{equation*}

Then $\phi$ is a weight function. 

We call weight functions of the above form \emph{semi-trivial weight functions}.
\end{prop}
\begin{proof}
By Proposition \ref{prop:integercond}, we only need to check that the identity
\begin{equation}
\label{eq:checkid}
\phi(\sigma,d_i\sigma)\cdot\phi(d_i\sigma,d_jd_i\sigma)=\phi(\sigma,d_j\sigma)\cdot\phi(d_j\sigma,d_jd_i\sigma)
\end{equation}
holds for all $\sigma\in K$ and all $j<i$.

We consider $d_i\sigma\in K$. Since $A\cup B=K$, we have either $d_i\sigma\in A$ or $d_i\sigma\in B$. Hence, either $\phi(d_i\sigma,d_jd_i\sigma)=0$ or $\phi(\sigma,d_i\sigma)=0$. In both cases, the left side of (\ref{eq:checkid}) is 0. Similarly, by considering $d_j\sigma$, the right side of (\ref{eq:checkid}) is also 0.
\end{proof}

\subsection{Dawson's Weighted Boundary Map}
\label{subsec:dawson}
In \cite{Dawson1990}, Dawson defined a \emph{weighted simplicial complex} to be a pair $(K,w)$ consisting of a simplicial complex $K$ and a function $w:K\to\mathbb{Z}$ satisfying a \emph{divisibility condition}: $\sigma_1\subseteq \sigma_2\implies w(\sigma_1)\mid w(\sigma_2)$ for all $\sigma_1,\sigma_2\in K$. 

The weighted boundary map $\tilde{\partial}$ \cite[p.~234]{Dawson1990} is then defined by
\begin{equation}
\label{eq:dawsonboundary}
\tilde{\partial}(\sigma)=\sum_{i}\frac{w(\sigma)}{w(d_i\sigma)}(-1)^id_i\sigma,
\end{equation}
where $\sigma$ is a simplex satisfying $w(\sigma)\neq 0$.

Dawson's definition of weighted boundary map (\ref{eq:dawsonboundary}) is a special case of our definition of the weight function $\phi$ and $\phi$-weighted boundary map. In fact, suppose $w:K\to\Z$ satisfies the divisibility condition, and also satisfies $w(\sigma)\neq 0$ for all $\sigma\in K$. Define $\phi:K\times K\to\Z$ by $\phi(\sigma,d_i\sigma)=\frac{w(\sigma)}{w(d_i\sigma)}\in\Z$ for all $\sigma$ and all $i$. Note that
\begin{equation*}
\phi(\sigma,d_i\sigma)\cdot\phi(d_i\sigma,d_jd_i\sigma)
=\frac{w(\sigma)}{w(d_jd_i\sigma)}=\phi(\sigma,d_j\sigma)\cdot\phi(d_j\sigma,d_jd_i\sigma).
\end{equation*}
Hence, by Proposition \ref{prop:integercond}, $\phi$ is a weight function.

\subsection{Example of a new nontrivial weight function}
We give an example to illustrate that a weight function $\phi$ can produce a $\phi$-weighted boundary map different from Dawson's weighted boundary map \cite{Dawson1990}.

\begin{eg}
Let $K=\{[v_0],[v_1],[v_2],[v_0,v_1],[v_0,v_2],[v_1,v_2],[v_0,v_1,v_2]\}$ be a 2-simplex. Let $\sigma=[v_0,v_1,v_2]$. Let $\phi: K\times K\to\Z$ be the weight function such that:
\begin{equation*}
\begin{alignedat}{25}
&&\phi(\sigma,d_0\sigma)=0, \qquad&&\phi(\sigma,d_1\sigma)=3, \qquad&&\phi(\sigma,d_2\sigma)=1,\\
&&\phi(d_0\sigma,d_0d_1\sigma)=2, \qquad &&\phi(d_0\sigma,d_0d_2\sigma)=4, \qquad&&\phi(d_1\sigma,d_0d_1\sigma)=0,\\
&&\phi(d_1\sigma,d_1d_2\sigma)=2, \qquad&&\phi(d_2\sigma,d_0d_2\sigma)=0, \qquad&&\phi(d_2\sigma,d_1d_2\sigma)=6.
\end{alignedat}
\end{equation*}

We can check that $\phi$ satisfies (\ref{eq:integercond}) hence is a weight function. We have 
\begin{equation*}
\partial_2([v_0,v_1,v_2])=0[v_1,v_2]-3[v_0,v_2]+1[v_0,v_1].
\end{equation*}

Note that the zero coefficient for $[v_1,v_2]$ immediately implies that our $\phi$-weighted boundary map is different from that in \cite{Dawson1990}. In \cite{Dawson1990}, all coefficients of $d_i\sigma$ are of the form $\frac{w(\sigma)}{w(d_i\sigma)}$ where $w(\sigma)\neq 0$. Hence it is impossible to have $\frac{w(\sigma)}{w(d_i\sigma)}=0$.
\end{eg}

\begin{prop}
Suppose a weight function $\phi: K\times K\to\mathbb{Z}$ can be expressed in the form $\phi(\sigma,d_i\sigma)=\frac{w(\sigma)}{w(d_i\sigma)}$ for some function $w:K\to\mathbb{Z}$ satisfying $\sigma_1\subseteq\sigma_2 \implies w(\sigma_1)\mid w(\sigma_2)$ for all $\sigma_1,\sigma_2\in K$. For convenience, we call such $\phi$ to be ``of Dawson type''.
Suppose $\phi(\sigma,d_i\sigma)\neq 0$ for all $\sigma\in K$ and that $\phi(\tau,d_0\tau)=1$ for all $0$-simplices $\tau$. Then $\phi$ is of Dawson type.
\end{prop}
\begin{proof}
We define $w:K\to\mathbb{Z}$ inductively on dimension, as follows.

For each $0$-simplex $\tau$, we define $w(\tau)=1$.

For each $1$-simplex $\sigma$, we define $w(\sigma)=\phi(\sigma,d_i\sigma)$. Proposition \ref{prop:integercond} and the convention that $\phi(d_i\sigma,d_jd_i\sigma)=1$ since $d_i\sigma$ is a $0$-simplex implies that $\phi(\sigma,d_i\sigma)\cdot 1=\phi(\sigma,d_j\sigma)\cdot 1$ for all 1-simplices $\sigma$ and for all $j<i$. Hence, $w$ is well-defined on 1-simplices. By construction, we have $\phi(\sigma,d_i\sigma)=\frac{w(\sigma)}{1}=\frac{w(\sigma)}{w(d_i\sigma)}$.

For $n\geq 2$, for each $n$-simplex $\sigma$, we define $w(\sigma)$ to be $\phi(\sigma,d_i\sigma)\cdot w(d_i\sigma)$. Using Proposition \ref{prop:integercond} and the inductive hypothesis we see that this is independent of $i$. Then, by construction we have $\phi(\sigma,d_i\sigma)=\frac{w(\sigma)}{w(d_i\sigma)}$ and $\sigma_1\subseteq\sigma_2 \implies w(\sigma_1)\mid w(\sigma_2)$ for all $\sigma_1,\sigma_2\in K$. In addition, since $\phi(\sigma,d_i\sigma)\neq 0$ for all $\sigma\in K$, hence $w(\sigma)\neq 0$ for all $\sigma\in K$ as well, ensuring that there is no division by zero.

Hence, $\phi$ is of Dawson type.
\end{proof}

\subsection{New way to construct weight function from a weighted simplicial complex $(K,w)$}
Let $K$ be a simplicial complex. Let $w:K\to\Z$ be an arbitrary function ($w$ does not need to satisfy the divisibility condition). Let $f:w(K)\to\Z\setminus\{0\}$ be an arbitrary function mapping values in $w(K)=\{w(\sigma)\mid \sigma\in K\}$ to nonzero integers. We define $\phi: K\times K\to\Z$ by
\begin{equation}
\label{eq:newnatural}
\phi(\sigma,d_i\sigma)=\frac{Cf(w(\sigma))}{f(w(d_i\sigma))},
\end{equation}
where $C\in\Z$ is a chosen constant such that $\frac{Cf(w(\sigma))}{f(w(d_i\sigma))}$ is always an integer for all $\sigma\in K$ and all $i$. (Such a $C$ always exists, for instance we can choose $C$ to be an arbitrary common multiple of $\{f(w(\sigma))\mid\sigma\in K\}$. Alternatively, we may choose $C=0$.) We remark that the $\phi$ defined in \eqref{eq:newnatural} is equivalent to Dawson's weighted boundary map induced by the weight function $w(\sigma)=C^{\dim(\sigma)}f(w(\sigma))$.

We can check that
\begin{equation*}
\phi(\sigma,d_i\sigma)\cdot\phi(d_i\sigma,d_jd_i\sigma)
=\frac{C^2f(w(\sigma))}{f(w(d_jd_i\sigma))}=\phi(\sigma,d_j\sigma)\cdot\phi(d_j\sigma,d_jd_i\sigma).
\end{equation*}
Hence, by Proposition \ref{prop:integercond}, $\phi$ is a weight function. We also note that when $C=1$, $f=\id$, by imposing the divisibility condition on $w:K\to\Z$, the weight function $\phi$ reduces to Dawson's weighted boundary map \cite{Dawson1990}.

We also show that the weight function $\phi$ as defined in Equation \ref{eq:newnatural} is the most general, in the sense that any other weight function constructed from a weighted simplicial complex has the same form as Equation \ref{eq:newnatural}.

\begin{defn}
Let $g:\Z\times \Z\to\Z$ satisfy the functional equation
\begin{equation}
\label{eq:functionalg}
g(x,y)g(y,z)=g(x,v)g(v,z)
\end{equation}
for all $v,x,y,z\in\Z$. We call $g(x,y)\equiv 0$ a \emph{trivial} solution. Let $A,B\subseteq\Z$ such that $A\cup B=\Z$. We call \[g(x,y)=\begin{cases}
0 &\text{if $x\in A$}\\
0 &\text{if $y\in B$}\\
\text{arbitrary} &\text{otherwise}
\end{cases}\]
a \emph{semi-trivial} solution. (We can verify that both the trivial and semi-trivial solutions do satisfy the functional equation (\ref{eq:functionalg}).)
\end{defn}

\begin{prop}
Let $g:\Z\times \Z\to\Z$ be a function that is not semi-trivial (which also implies $g$ is not trivial). Suppose that $\phi':K\times K\to\Z$ where 
\begin{equation*}
\phi'(\sigma,d_i\sigma)=g(w(\sigma),w(d_i\sigma))
\end{equation*}
is a weight function for all functions $w:K\to\Z$.

Then $\phi'$ must be of the form \[\phi'(\sigma,d_i\sigma)=\frac{ch(w(\sigma))}{h(w(d_i\sigma))},\] where $c$ is a constant and $h:\Z\to\Z\setminus\{0\}$ is a function. That is, $\phi'$ must have the same form as Equation \ref{eq:newnatural}.
\end{prop}
\begin{proof}
By Proposition \ref{prop:integercond}, $g$ must satisfy 
\begin{equation*}
g(w(\sigma),w(d_i\sigma))\cdot g(w(d_i\sigma),w(d_jd_i\sigma))=g(w(\sigma),w(d_j\sigma))\cdot g(w(d_j\sigma),w(d_jd_i\sigma))
\end{equation*}
for all $\sigma\in K$, all $j<i$ and all $w:K\to\Z$.

Since $w:K\to\Z$ has no restrictions, $w(\sigma)$ can take any value in $\Z$. Hence we may extend the requirement on $g$ to (\ref{eq:functionalg}), that is, 
\begin{equation*}
g(x,y)g(y,z)=g(x,v)g(v,z)
\end{equation*}
for all $v,x,y,z\in \Z$.

Next, we show that if $g$ is neither a trivial nor semi-trivial solution, then $g$ must be of the form $g(x,y)=\frac{ch(x)}{h(y)}$, where $c$ is a constant and $h:\Z\to\Z\setminus\{0\}$ is a function.

Firstly, we note that if for some $a\in\Z$, $g(a,\cdot)\equiv 0$, then since $g(x,y)g(y,z)=g(x,a)g(a,z)$ for all $x,y,z$, we have that for all $y$ either $g(\cdot,y)\equiv 0$ or $g(y,\cdot)\equiv 0$. Let $A=\{y\mid g(y,\cdot)\equiv 0\}$, $B=\{y\mid g(\cdot,y)\equiv 0\}$. We see that $A\cup B=\Z$ and $g$ is a semi-trivial solution. Similarly, if for some $b\in\Z$ we have $g(\cdot,b)\equiv 0$, then $g$ is necessarily semi-trivial.

Suppose $g$ is not semi-trivial. Let $a\in\Z$ and define $h(x):=g(x,a)$. Since $g(\cdot,a)$ is not identically zero, there exists $b$ such that $h(b)\neq 0$. We have
\begin{equation*}
g(x,y)h(y)=g(x,y)g(y,a)=g(x,a)g(a,a)=h(x)h(a).
\end{equation*}

Suppose $h(y)=0$ for some $y$. Then $h(x)h(a)=0$ for all $x$, in particular $h(a)=0$. Then $g(\cdot,b)h(b)\equiv 0$ implies $g(\cdot,b)\equiv 0$. This contradicts that $g$ is not semi-trivial. Thus, $h(y)\neq 0$ for all $y$ and hence \[g(x,y)=\frac{h(x)h(a)}{h(y)}=\frac{ch(x)}{h(y)},\] where $c:=h(a)$.

We have shown that all non-trivial $g$ satisfying Equation \ref{eq:functionalg} has the form $g(x,y)=\frac{ch(x)}{h(y)}$, where $c$ is a constant and $h:\Z\to\Z\setminus\{0\}$ is a function.

This implies that $\phi'(\sigma,d_i\sigma)=\frac{ch(w(\sigma))}{h(w(d_i\sigma))}$.
\end{proof}

\section{$\phi$-Weighted Homology and Cohomology}
Let $G$ be an $R$-module, where $R$ is a commutative ring with unity. Let $K$ be a simplicial complex and $\phi:K\times K\to R$ be a weight function. Let $\partial_n: C_n(K;G)\to C_{n-1}(K;G)$ be the $n$-th $\phi$-weighted boundary map. (Note that the definition of $\partial_n$ depends on the choice of $\phi$.)

When $G=\mathbb{Z}$, we omit $G$ from the notation. For instance, we write $C_n(K)=C_n(K;\Z)$.
\begin{defn}
We define the \emph{$n$-th $\phi$-weighted homology group} of $K$, with coefficients in $G$, to be
\begin{equation*}
H_n(K,\phi;G)=\ker\partial_n/\Ima\partial_{n+1},
\end{equation*}
where $\partial_n$ is the $n$-th $\phi$-weighted boundary map.
\end{defn}

\begin{remark}
We emphasize in the notation $H_n(K,\phi;G)$ that the weighted homology group depends on the weight function $\phi$. In general, different weight functions may lead to different weighted homology.
\end{remark}

\begin{defn}
\label{defn:cochaingroup}
Let $R$ be a commutative ring with unity and $G$ be an $R$-module. The group of \emph{$n$-dimensional cochains} of $K$, with coefficients in $G$, is the group
\begin{equation*}
C^n(K;G)=\Hom_R(C_n(K;R),G).
\end{equation*}
\end{defn}

\begin{remark}
By the tensor-hom adjunction, we have
\begin{equation*}
\begin{split}
\Hom_R(C_n(K;R),G)&\cong\Hom_R(C_n(K)\otimes_\Z R,G)\\
&\cong\Hom_\Z(C_n(K),\Hom_R(R,G))\\
&\cong\Hom_\Z(C_n(K),G).
\end{split}
\end{equation*}
Hence, we note that Definition \ref{defn:cochaingroup} is independent of the choice of the ring $R$. Nevertheless, we keep the definition with emphasis on $R$ so as to allow more general weight functions $\phi:K\times K\to R$, instead of restricting to weight functions $\phi:K\times K\to\Z$.
\end{remark}

\begin{defn}
Let $\phi:K\times K\to R$ be a weight function and $\partial$ be the $\phi$-weighted boundary operator. The \emph{$\phi$-weighted coboundary operator} $\delta_n: C^n(K;G)\to C^{n+1}(K;G)$ is defined to be the dual of the $\phi$-weighted boundary operator $\partial_{n+1}: C_{n+1}(K;R)\to C_n(K;R)$. That is, if $f: C_n(K;R)\to G$ is a $n$-dimensional cochain, then
\begin{equation*}
\delta_n(f)=f\circ\partial_{n+1}.
\end{equation*}

Suppose $c_n\in C_n(K;R)$ and $f\in C^n(K;G)$. We commonly use the notation $\langle f,c_n\rangle$ to denote $f(c_n)$. With this notation, the definition of the weighted coboundary operator becomes
\begin{equation}
\label{eqn:coboundary}
\langle\delta_n f,c_{n+1}\rangle=\langle f,\partial_{n+1}c_{n+1}\rangle.
\end{equation}
\end{defn}

\begin{prop}
Let $\phi: K\times K\to R$ be the weight function, and let $\partial_n: C_n(K;R)\to C_{n-1}(K;R)$ be the $\phi$-weighted boundary map. The formula for the $\phi$-weighted coboundary operator $\delta$ is 
\begin{equation*}
\langle\delta_nf,\sigma\rangle=\sum_{i=0}^{n+1}(-1)^i\phi(\sigma,d_i\sigma)\langle f,d_i\sigma\rangle,
\end{equation*}
where $f\in C^n(K;G)$, $\sigma=[v_0,\dots,v_{n+1}]$ is a $(n+1)$-simplex, and $d_i$ is the $i$-th face map.
\end{prop}
\begin{proof}
The proof follows directly from (\ref{eqn:coboundary}). To be precise, we have
\begin{equation*}
\begin{split}
\langle\delta_n f,\sigma\rangle&=\langle f,\sum_{i=0}^{n+1}(-1)^i\phi(\sigma,d_i\sigma)d_i\sigma\rangle\\
&=\sum_{i=0}^{n+1}(-1)^i\phi(\sigma,d_i\sigma)\langle f,d_i\sigma\rangle.
\end{split}
\end{equation*}
The last equality follows from the fact that $f$ is a $R$-module homomorphism.
\end{proof}

\begin{defn}
We define the $n$-th $\phi$-weighted cohomology group of $K$, with coefficients in $G$, to be
\[H^n(K,\phi;G)=\ker\delta_n/\Ima\delta_{n-1},\]
where $\delta_n$ is the $n$-th $\phi$-weighted coboundary operator.
\end{defn}

\begin{remark}
Note that $\delta_{n}=0$, $C^n(K;G)=0$ for $n\leq -1$. Hence, $H^{n}(K,\phi;G)=0$ for $n\leq -1$.
\end{remark}


\subsection{Difference between $\phi$-Weighted Cohomology Group and Usual Cohomology Group with Real or Complex Coefficients}
Let $\F=\R$ or $\C$. Let $\phi:K\times K\to\F$ be a weight function. We study the difference between the $\phi$-weighted cohomology group ${H}^n(K,\phi;\F)$ and the usual cohomology group ${H}^n(K;\F)$.


\begin{remark}
We observe that Proposition \ref{prop:integercond} still holds for $G=\F$, $R=\F$. Hence, a function $\phi:K\times K\to\F$ is a weight function if and only if (\ref{eq:integercond}) holds for all $\sigma\in K$ and all $j<i$.
\end{remark}

We give an example to show that ${H}^n(K,\phi;\F)$ and ${H}^n(K;\F)$ may be different.

\begin{eg}
\label{eg:diffweighted}
Consider $K=\{[v_0],[v_1],[v_2],[v_0,v_1],[v_0,v_2],[v_1,v_2]\}$. Let $\phi:K\times K\to\F$ be the weight function such that:
\begin{equation*}
\begin{alignedat}{5}
&&\phi([v_0,v_1],[v_0])=2, \qquad&&\phi([v_0,v_1],[v_1])=1, \qquad&&\phi([v_1,v_2],[v_1])=1,\\
&&\phi([v_1,v_2],[v_2])=1, \qquad&&\phi([v_0,v_2],[v_0])=1, \qquad&&\phi([v_0,v_2],[v_2])=1.
\end{alignedat}
\end{equation*}
The general 0-cochain is of the form $c^0=r_0[v_0]^*+r_1[v_1]^*+r_2[v_2]^*$, where $r_0,r_1,r_2\in\F$. We calculate that 
\begin{equation*}
\begin{split}
\langle\delta_0 c^0,[v_0,v_1]\rangle&=\phi([v_0,v_1],[v_1])\langle c^0,[v_1]\rangle-\phi([v_0,v_1],[v_0])\langle c^0,[v_0]\rangle\\
&=r_1-2r_0.
\end{split}
\end{equation*}

Similarly, $\langle\delta_0c^0,[v_0,v_2]\rangle=r_2-r_0$ and $\langle\delta_0c^0,[v_1,v_2]\rangle=r_2-r_1$. Hence, if $c^0\in\ker\delta_0$, then $r_2=r_0=r_1$ and $r_1=2r_0$. This implies that $r_0=r_1=r_2=0$. Hence $\ker\delta_0\cong 0$.

We conclude that the 0-th $\phi$-weighted cohomology group is ${H}^0(K,\phi;\F)\cong 0$. However since $K$ is path-connected, the 0-th usual cohomology group is ${H}^0(K;\F)\cong\F$.
\end{eg}

\begin{prop}
\label{prop:dimH}
Let $\delta$ denote the $\phi$-weighted coboundary operator. Let $\dim V:=\dim_\F V$ denote the dimension of an $\F$-vector space $V$. We have that 
\begin{equation*}
\dim{H}^n(K,\phi;\F)=\dim C^n(K;\F)-\rank\delta_n-\rank\delta_{n-1}.
\end{equation*}
\end{prop}
\begin{proof}
We have
\begin{equation*}
\begin{split}
\dim{H}^n(K,\phi;\F)&=\dim\ker\delta_n-\dim\Ima\delta_{n-1}\\
&=\dim C^n(K;\F)-\dim\Ima\delta_n-\dim\Ima\delta_{n-1}\\
&=\dim C^n(K;\F)-\rank\delta_n-\rank\delta_{n-1}.
\end{split}
\end{equation*}
\end{proof}

\section{$\phi$-Weighted Combinatorial Laplace Operator ($\phi$-Weighted Laplacian)}
\label{sec:laplace}
The Laplacian (also known as the Laplace operator) was first studied on graphs, and later generalized to simplicial complexes by Eckmann \cite{eckmann1944harmonische}. The Laplacian on simplicial complexes was also studied by Horak and Jost \cite{Horak2013}, and by Muhammad and Egerstedt \cite{muhammad2006control}. Laplacians have applications in networked control and sensing \cite{muhammad2006control} and statistical ranking \cite{jiang2011statistical}. In this section, we generalize the Laplacian on simplicial complexes by using the $\phi$-weighted coboundary map defined previously. In addition, we generalize to the case of complex coefficients. 

Let $\F=\R$ or $\C$. We let $G=\F$, $R=\F$ in this section, that is, both the chain group $C_n(K;G)$ and cochain group $C^n(K;G)$ have real or complex coefficients. Also, $G=\F$ is considered as a $\F$-module (vector space). If $z$ is a complex number, we denote by $\bar{z}$ the complex conjugate of $z$.

Let $K$ be a simplicial complex, with weight function $\phi:K\times K\to\F$. We choose inner products $(\cdot,\cdot)_{C^n}$ and $(\cdot,\cdot)_{C^{n+1}}$ on the vector spaces $C^n(K;\F)$ and $C^{n+1}(K;\F)$ respectively.

\begin{eg}[Example of inner product]
\label{eg:standardinner}
Let $f, g\in C^n(K;\F)$. Let $\{\sigma_1^*,\sigma_2^*,\dots,\sigma_m^*\}$ be a basis for $C^n(K;\F)$. Then $f=\sum_{i=1}^m r_i\sigma_i^*$ and $g=\sum_{i=1}^m s_i\sigma_i^*$ for some $r_i, s_i\in\F$. We define $(\cdot,\cdot)_{C^n}: C^n(K;\F)\times C^n(K;\F)\to\F$ by 
\begin{equation*}
(f,g)_{C^n}:=\sum_{i=1}^m r_i \overline{s_i}.
\end{equation*}

We can verify that $(\cdot,\cdot)_{C^n}$ is indeed an inner product, which we call the \emph{standard inner product} on $C^n(K;\F)$.
\end{eg}

\begin{defn}
The \emph{adjoint} $\delta_n^*: C^{n+1}(K;\F)\to C^n(K;\F)$ of the weighted coboundary operator $\delta_n$ is defined by
\begin{equation}
(\delta_n f,g)_{C^{n+1}}=(f,\delta_n^*g)_{C^n},
\end{equation}
for all $f\in C^n(K;\F)$ and $g\in C^{n+1}(K;\F)$.
\end{defn}

\subsection{Matrix Representation of the Adjoint}
\begin{defn}
\label{defn:standardbasis}
Let $\F=\R$ or $\C$. Let $B_n=\{\sigma_1,\sigma_2,\dots,\sigma_m\}$ be the (ordered) set of $n$-simplices of $K$, which forms a basis for $C_n(K;\F)$. We call $B_n$ the \emph{standard ordered basis} for $C_n(K;\F)$. Let $B_{n+1}=\{\tau_1,\tau_2,\dots,\tau_l\}$ be the standard ordered basis for $C_{n+1}(K;\F)$. Then, with respect to the above ordered bases, the $\phi$-weighted boundary map $\partial_{n+1}$ can be represented as a $m\times l$ matrix with entries in $\F$, denoted as $[\partial_{n+1}]$.

Let $B_n^*=\{\sigma_1^*,\sigma_2^*,\dots,\sigma_m^*\}$ be the \emph{standard ordered dual basis} for $C^n(K;\F)$. Let $B_{n+1}^*=\{\tau_1^*,\tau_2^*,\dots,\tau_l^*\}$ be the standard ordered dual basis for $C^{n+1}(K;\F)$. The $\phi$-weighted coboundary map $\delta_n$ can be represented as a $l\times m$ matrix, denoted as $[\delta_n]$. Similarly, the adjoint $\delta_n^*$ with respect to the standard inner product, can be represented as a $m\times l$ matrix $[\delta_n^*]$.
\end{defn}

\begin{remark}
In this paper, the square bracket $[\partial]$ (resp.\ $[\delta]$) will refer to the matrix representation of a boundary operator (resp.\ coboundary operator) with respect to the standard ordered basis (resp.\ dual basis) as described in Definition \ref{defn:standardbasis}. Also, unless otherwise specified, $[\delta_n^*]$ will refer to the matrix representation of the adjoint $\delta_n^*$ with respect to the standard inner product. 
\end{remark}

\begin{prop}
\label{prop:matrixdelta}
Let $\partial$ and $\delta$ be the $\phi$-weighted boundary and coboundary operator respectively. Let $(\cdot,\cdot)_{C^n}$ be the standard inner product on $C^n(K;\F)$.
We have the following relationships between the matrix representations (with respect to the standard ordered bases):
\begin{align*}
[\delta_n]&=[\partial_{n+1}]^T\\
[\delta_n^*]&=[\delta_n]^\dagger,
\end{align*}
where $A^T$ and $A^\dagger$ denote the transpose and conjugate transpose of a matrix $A$ respectively.
\end{prop}
\begin{proof}
Write $[\delta_n]=[\alpha_{ij}]$. Then we have $\delta_n\sigma_j^*=\sum_{i=1}^l \alpha_{ij}\tau_i^*$ for $1\leq j\leq m$.

Note that since $B_{n+1}^*$ is an orthonormal basis, we have $(\delta_n\sigma_j^*,\tau_i^*)_{C^{n+1}}=\alpha_{ij}$ for $1\leq i\leq l$, $1\leq j\leq m$. 

Write $[\delta_n^*]=[\beta_{ij}]$. Similarly, we have $\delta_n^*\tau_j^*=\sum_{i=1}^m\beta_{ij}\sigma_i^*$ for $1\leq j\leq l$. Since $B_n^*$ is also an orthonormal basis, we have $(\delta_n^*\tau_j^*,\sigma_i^*)_{C^n}=\beta_{ij}$
for $1\leq i\leq m$, $1\leq j\leq l$.

Then by the definition of the adjoint,
\begin{equation*}
\beta_{ij}=\overline{(\sigma_i^*,\delta_n^*\tau_j^*)}_{C^n}=\overline{(\delta_n\sigma_i^*,\tau_j^*)}_{C^{n+1}}=\overline{\alpha_{ji}}.
\end{equation*}
Hence $[\delta_n^*]=[\delta_n]^\dagger$.

Next, we write $[\partial_{n+1}]=[\gamma_{ij}]$. By the definitions of $[\delta_n]$ and $[\partial_{n+1}]$ we have
\begin{equation*}
\alpha_{ij}=(\delta_n\sigma_j^*)(\tau_i)=\sigma_j^*(\partial_{n+1}(\tau_i))=\gamma_{ji}.
\end{equation*}
Hence $[\delta_n]^T=[\partial_{n+1}]$.
\end{proof}

\subsection{$\phi$-Weighted Laplace operator}
\begin{defn}
\label{defn:weightedlap}
Let $\F=\R$ or $\C$. We define the \emph{$n$-dimensional $\phi$-weighted Laplace operator} (also called the \emph{$\phi$-weighted Laplacian}) $\Delta_n: C^n(K;\F)\to C^n(K;\F)$ by
\begin{equation*}
\Delta_n=\delta_{n-1}\delta_{n-1}^*+\delta_n^*\delta_n,
\end{equation*}
where $\delta_n$ is the $\phi$-weighted coboundary operator, and $\delta_n^*$ is the adjoint taken relative to a chosen inner product.
\end{defn}

We now prove a generalization of the discrete Hodge theorem (cf.\ \cite{eckmann1944harmonische},\cite[p.~308]{Horak2013}) to the weighted case.
\begin{theorem}
\label{thm:mainiso}
Let $K$ be a  simplicial complex, $\phi:K\times K\to\F$ be a weight function, and $\Delta_n$ be the $n$-dimensional $\phi$-weighted Laplace operator. Then, we have
\begin{equation*}
\ker\Delta_n\cong{H}^n(K,\phi;\F).
\end{equation*}
\end{theorem}
\begin{proof}
Since $\delta_n\delta_{n-1}=0$ and $\delta_{n-1}^*\delta_n^*=(\delta_n\delta_{n-1})^*=0$, thus $\Ima\delta_{n-1}\delta_{n-1}^*\subseteq\ker\delta_n^*\delta_n$ and $\Ima\delta_n^*\delta_n\subseteq\ker\delta_{n-1}\delta_{n-1}^*$.

Note that
\begin{equation*}
\ker\delta_{n-1}^*\cap\ker\delta_n\subseteq\ker\delta_{n-1}\delta_{n-1}^*\cap\ker\delta_{n}^*\delta_n\subseteq\ker\Delta_n.
\end{equation*}

Conversely, if $f\in\ker\Delta_n$, then
\begin{equation*}
\begin{split}
0&=(\Delta_nf,f)_{C^n}\\
&=(\delta_{n-1}\delta_{n-1}^*f,f)_{C^n}+(\delta_n^*\delta_nf,f)_{C^n}\\
&=(\delta_{n-1}^*f,\delta_{n-1}^*f)_{C^n}+(\delta_nf,\delta_nf)_{C^n}.
\end{split}
\end{equation*}
By positive-definiteness of inner product, we have $\delta_{n-1}^*f=\delta_n f=0$. Hence $f\in\ker\delta_{n-1}^*\cap\ker\delta_n$. We have shown $\ker\Delta_n\subseteq\ker\delta_{n-1}^*\cap\ker\delta_n$.

Hence,
\begin{equation}
\label{eqn:perpcong}
\begin{split}
\ker\Delta_n&=\ker\delta_{n-1}^*\cap\ker\delta_n\\
&=(\Ima\delta_{n-1})^\perp\cap\ker\delta_n\\
&\cong{H}^n(K,\phi;\F).
\end{split}
\end{equation}

The isomorphism in the last line of (\ref{eqn:perpcong}) can be constructed via the isomorphism
\begin{equation*}
\begin{split}
\Psi: (\Ima\delta_{n-1})^\perp\cap\ker\delta_n&\to{H}^n(K,\phi,\F)\\
f&\mapsto f+\Ima\delta_{n-1}.
\end{split}
\end{equation*}
\end{proof}

\subsection{Matrix Representation of the $\phi$-Weighted Laplacian}
\label{subsec:matweightlap}
In this section, we study the matrix representation of the $\phi$-weighted Laplacian $\Delta$ defined in Definition \ref{defn:weightedlap}.

If we adopt the approach of giving weight to the inner product \cite[p.~309]{Horak2013}, our $\phi$-weighted Laplacian generalizes \cite{Horak2013}. In particular, our $\phi$-weighted Laplacian reduces to the one in \cite{Horak2013} when $\phi$ is the identity weight function $\phi(\sigma, d_i\sigma)\equiv 1$ (cf.\ Example \ref{eg:id}).

\begin{prop}
\label{prop:ourupdown}
Let $A_i$ be the matrix corresponding to the $\phi$-weighted coboundary operator $\delta_i$ and let $A_i^\dagger$ be its conjugate transpose.

Then we have the following matrices corresponding to the respective operators:
\begin{align*}
[\delta_i^*\delta_i]&=A_i^\dagger A_i,\\
[\delta_{i-1}\delta_{i-1}^*]&=A_{i-1}A_{i-1}^\dagger,\\
[\Delta_i]&=A_{i-1}A_{i-1}^\dagger+A_i^\dagger A_i.
\end{align*}
\end{prop}
\begin{proof}
The proof follows from Proposition \ref{prop:matrixdelta} and Definition \ref{defn:weightedlap}.
\end{proof}

We generalize the following proposition, which is stated but not proved in \cite[p.~309]{Horak2013}, to complex numbers.
\begin{prop}[cf.\ {\cite[p.~309]{Horak2013}}]
\label{prop:weightedinner}
Let $\F=\R$ or $\C$. Let $w:K\to\R^+$ be a function. We say that the weight of a simplex $\sigma\in K$ is $w(\sigma)$. For any choice of inner product on the cochain group $C^i(K,\F)$, where elementary cochains form an orthogonal basis, there exists a function $w$ such that
\begin{equation*}
(f,g)_{C^i}=\sum_\sigma w(\sigma)f(\sigma)\overline{g(\sigma)},
\end{equation*}
where $\sigma$ runs over the $i$-simplices of $K$. Furthermore, there is a one-to-one correspondence between functions $w:K\to\R^+$ and possible inner products on $C^i(K,\F)$ where elementary cochains are orthogonal.
\end{prop}
\begin{proof}
Let $\{\sigma_1^*,\dots,\sigma_m^*\}$ be the standard ordered dual basis for $C^i(K;\F)$. We write $f=\sum_{j=1}^m r_j\sigma_j^*$ and $g=\sum_{j=1}^m s_j\sigma_j^*$. Then \[(f,g)_{C^i}=\sum_{j=1}^m r_j\overline{s_j}(\sigma_j^*,\sigma_j^*)_{C^i},\] by orthogonality of the elementary cochains. If we define $w(\sigma)=(\sigma^*,\sigma^*)_{C^i}>0$, then \[(f,g)_{C^i}=\sum_\sigma w(\sigma)f(\sigma)\overline{g(\sigma)},\] where $\sigma$ runs over the $i$-simplices of $K$. Clearly, our choice of $w$ gives the desired one-to-one correspondence.
\end{proof}

\begin{remark}
When $w(\sigma)\equiv 1$ for all $\sigma\in K$, the inner product in Proposition \ref{prop:weightedinner} corresponds to the standard inner product (cf.\ Example \ref{eg:standardinner}), where the elementary cochains are \emph{orthonormal}.
\end{remark}

We now state the corresponding result of Proposition \ref{prop:ourupdown} from \cite{Horak2013}. To avoid confusion, we will let $\widetilde{\delta_i}$ denote the usual unweighted coboundary operator. We show that the proposition holds for cochain groups $C^i(K,\F)$, where $\F=\R$ or $\C$.

\begin{prop}[cf.\ {\cite[p.~310]{Horak2013}}]
\label{prop:theirupdown}
Let ${D_i}$ be the matrix corresponding to the usual unweighted coboundary operator $\widetilde{\delta_i}: C^i(K;\F)\to C^{i+1}(K;\F)$. Let $W_i$ be the \emph{diagonal} matrix representing the scalar product on $C^i$, then the $\mathcal{L}_i^{up}$ and $\mathcal{L}_i^{down}$ operators are expressed as:
\begin{align*}
[\mathcal{L}_i^{up}]&=W_i^{-1}{D_i^T}W_{i+1}{D_i},\\
[\mathcal{L}_i^{down}]&={D_{i-1}}W_{i-1}^{-1}{D_{i-1}^T}W_i.
\end{align*}

The matrix for the weighted $i$-dimensional combinatorial Laplace operator is then expressed as
\begin{equation*}
[\mathcal{L}_i]=[\mathcal{L}_i^{up}]+[\mathcal{L}_i^{down}].
\end{equation*}
\end{prop}
\begin{proof}
Let $B_i^*=\{\sigma_1^*,\dots,\sigma_m^*\}$ be the standard ordered dual basis for $C^i(K;\F)$. Let $B_{i+1}^*=\{\tau_1^*,\dots,\tau_l^*\}$ be the standard ordered dual basis for $C^{i+1}(K;\F)$.

Write $[\widetilde{\delta_i}^*]=[\beta_{jk}]$. We have $\widetilde{\delta_i}^*\tau_k^*=\sum_{j=1}^m\beta_{jk}\sigma_j^*$. Since $B_i^*$ is an orthogonal basis, we have
\begin{equation*}
(\widetilde{\delta_i}^*\tau_k^*,\sigma_j^*)_{C^i}=\beta_{jk}(\sigma_j^*,\sigma_j^*)_{C^i}=\beta_{jk}w(\sigma_j).
\end{equation*}

Write $D_i=[\alpha_{jk}]$. Similarly, we have $\widetilde{\delta_i}\sigma_k^*=\sum_{j=1}^l\alpha_{jk}\tau_j^*$ and $(\widetilde{\delta_i}\sigma_k^*,\tau_j^*)_{C^{i+1}}=\alpha_{jk}w(\tau_j)$.

Hence,
\begin{equation}
\label{eq:walphareal}
\begin{split}
\beta_{jk}&=\frac{1}{w(\sigma_j)}(\widetilde{\delta_i}^*\tau_k^*,\sigma_j^*)_{C^i}\\
&=\frac{1}{w(\sigma_j)}\overline{(\widetilde{\delta_i}\sigma_j^*,\tau_k^*)_{C^{i+1}}}\\
&=\frac{\overline{w(\tau_k)}}{w(\sigma_j)}\overline{\alpha_{kj}}\\
&=\frac{w(\tau_k)}{w(\sigma_j)}\alpha_{kj}.
\end{split}
\end{equation}
The last equality in (\ref{eq:walphareal}) follows since $w(\tau_k)$ is real and the matrix $[\alpha_{jk}]$ consists only of entries in $\{-1,0,1\}$.

This implies that $[\widetilde{\delta_i}^*]=W_i^{-1}D_i^TW_{i+1}$, where $W_i=\diag(w(\sigma_1),\dots,w(\sigma_m))$ and $W_{i+1}=\diag(w(\tau_1),\dots,w(\tau_l))$. Hence
\begin{equation*}
[\mathcal{L}_i^{up}]=[\widetilde{\delta_i}^*][\widetilde{\delta_i}]=W_i^{-1}D_i^TW_{i+1}D_i.
\end{equation*}
The case for $[\mathcal{L}_i^{down}]$ is similar.
\end{proof}

Let $\delta$ denote the $\phi$-weighted coboundary map. It is clear that $[\delta_i^*\delta_i]$ in Proposition \ref{prop:ourupdown} is the counterpart of $[\mathcal{L}_i^{up}]$ in Proposition \ref{prop:theirupdown}. Similarly, $[\delta_{i-1}\delta_{i-1}^*]$ is the counterpart of $[\mathcal{L}_i^{down}]$. We show by an example that even in the real case, it may not be possible to write $[\delta_i^*\delta_i]$ in the form $W_i^{-1}D_i^TW_{i+1}D_i$, hence our definition of the weighted Laplacian is new and different from \cite{Horak2013}. (For the complex case, it is clear that there is a difference since the matrix $[\delta_i^*\delta_i]$ can have complex entries while the matrix $W_i^{-1}D_i^TW_{i+1}D_i$ only has real entries.)

\begin{eg}
\label{eg:newlaplacian1}
Consider the 1-simplex $K=\{[v_0],[v_1],[v_0,v_1]\}$. Let $\phi:K\times K\to\R$ be a weight function such that $\phi([v_0,v_1],[v_0])=p$ and $\phi([v_0,v_1],[v_1])=q$, where $p,q\in\R$.

Let $p=2$, $q=3$. Then, $\partial_1([v_0,v_1])=3[v_1]-2[v_0]$. Hence, $[\partial_1]=\begin{bmatrix}-2\\3\end{bmatrix}$. By Proposition \ref{prop:matrixdelta}, $[\delta_0]=[\partial_1]^T=\begin{bmatrix}-2 &3\end{bmatrix}$ and $[\delta_0^*]=[\delta_0]^\dagger=\begin{bmatrix}-2\\3\end{bmatrix}$.

Let $D_i$ be the matrix corresponding to the usual unweighted coboundary operator $\widetilde{\delta_i}$. Then, $D_0=\begin{bmatrix}-1 &1\end{bmatrix}$ and $D_0^T=\begin{bmatrix}-1\\1\end{bmatrix}$.

Let $W_0^{-1}=\begin{bmatrix}a & 0\\0 & b\end{bmatrix}$ and $W_1=\begin{bmatrix}c\end{bmatrix}$.

Then, $[\delta_0^*\delta_0]=\begin{bmatrix}4 & -6\\-6 &9\end{bmatrix}$ and 
$[\mathcal{L}_0^{up}]=\begin{bmatrix}ac & -ac\\
-bc & bc
\end{bmatrix}$.

It is clear that $[\delta_0^*\delta_0]\neq[\mathcal{L}_0^{up}]$ for all $a,b,c\in\R$. This shows that $[\delta_0^*\delta_0]$ cannot be expressed in the form $W_0^{-1}D_0^TW_1D_0$ where $W_i$ is a diagonal matrix.
\end{eg}

\begin{eg}
We show that the counterexample in Example \ref{eg:newlaplacian1} still holds when $W_0$ and $W_1$ are allowed to be any symmetric positive definite matrices.

We note that if $W_0$ is symmetric positive definite, then $W_0$ is invertible and $W_0^{-1}$ is symmetric. 

Suppose $W_0^{-1}=\begin{bmatrix}a &d\\d & b\end{bmatrix}$ and $W_1=\begin{bmatrix}c\end{bmatrix}$. 

Then, $[\mathcal{L}_0^{up}]=W_0^{-1}D_0^TW_1D_0=\begin{bmatrix}-c(-a+d) & c(-a+d)\\-c(-d+b) & c(-d+b)\end{bmatrix}$.

Since $[\delta_0^*\delta_0]=\begin{bmatrix}4 &-6\\-6 &9\end{bmatrix}$, we note that $[\delta_0^*\delta_0]\neq [\mathcal{L}_0^{up}]$ for all $a,b,c\in\mathbb{R}$ since $4$ and $-6$ do not differ by a multiple of $-1$ but $-c(-a+d)$ and $c(-a+d)$ do.
\end{eg}

Next, we attempt to generalize \cite{Horak2013}. We note that if we just consider the $\phi$-weighted coboundary operator $\delta$ with the standard inner product, it is insufficient to generalize \cite{Horak2013}. That is, there exists $[\mathcal{L}_i^{up}]=W_i^{-1}D_i^TW_{i+1}D_i$ that cannot be written as $[\delta_i^*\delta_i]$ where $\delta_i$ is the $\phi$-weighted coboundary operator and $\delta_i^*$ is the adjoint corresponding to the standard inner product. We show this in the following example.

\begin{eg}
We use the same example as in Example \ref{eg:newlaplacian1}. Recall that in Example \ref{eg:newlaplacian1},
$[\mathcal{L}_0^{up}]=W_0^{-1}D_0^TW_1D_0=\begin{bmatrix}ac &-ac\\
-bc & bc\end{bmatrix}$.

Similar to Example \ref{eg:newlaplacian1}, we can calculate that
$[\delta_0^*\delta_0]=\begin{bmatrix}
p^2 &-pq\\
-pq & q^2
\end{bmatrix}$.

If we choose $a\neq b$ and $c\neq 0$, we see that $-ac\neq -bc$ hence $[\mathcal{L}_0^{up}]$ cannot be equal to $[\delta_0^*\delta_0]$.
\end{eg}

If we use the weighted inner product as described in Proposition \ref{prop:weightedinner}, together with our $\phi$-weighted coboundary operator, it turns out that we can generalize \cite{Horak2013}. We show this in Proposition \ref{prop:genhorak}.

\begin{prop}[Matrix Representation of $\phi$-Weighted Laplacian with Weighted Inner Product]
Let $\F=\R$ or $\C$. Let $W_i$ be the diagonal matrix representing the weighted inner product on $C^i(K;\F)$. Then the matrix representations of the $\phi$-weighted Laplacian operators with respect to the standard basis on $C^i(K;\F)$ are expressed as:
\begin{align*}
[\Delta_i^{up}&]=W_i^{-1}[\delta_i]^\dagger W_{i+1}[\delta_i]\\
[\Delta_i^{down}&]=[\delta_{i-1}]W_{i-1}^{-1}[\delta_{i-1}]^\dagger W_i,
\end{align*}
where $\delta_i$ refers to the $i$-th $\phi$-weighted coboundary operator.

The matrix representation of the $\phi$-weighted $i$-dimensional Laplacian is then expressed as
\begin{equation*}
[\Delta_i]=[\Delta_i^{up}]+[\Delta_i^{down}].
\end{equation*}
\end{prop}
\begin{proof}
The proof is similar to that of Proposition \ref{prop:theirupdown}. We remark that $[\delta_n]$ can have complex entries, and we have used the fact that $[\delta_n^*]=[\delta_n]^\dagger$ (Proposition \ref{prop:matrixdelta}).
\end{proof}

\begin{prop}
\label{prop:genhorak}
Let $\phi:K\times K\to\F$ be the identity weight function, i.e.\ $\phi(\sigma,d_i\sigma)\equiv 1$ for all $\sigma\in K$ and all $i$. Let $\delta$ denote the $\phi$-weighted coboundary operator. Then,
\begin{align*}
[\Delta_i^{up}]&=[\mathcal{L}_i^{up}]=W_i^{-1}D_i^TW_{i+1}D_i\\
[\Delta_i^{down}]&=[\mathcal{L}_i^{down}]=D_{i-1}W_{i-1}^{-1}D_{i-1}^TW_i,
\end{align*}
where $D_i$ is the matrix corresponding to the usual unweighted coboundary operator $\widetilde{\delta_i}$.
\end{prop}
\begin{proof}
When $\phi$ is the identity weight function, the $\phi$-weighted coboundary operator $\delta_i$ reduces to the usual unweighted coboundary operator $\widetilde{\delta_i}$. Hence $[\delta_i]=D_i$ and $[\delta_i^*]=D_i^T$. Thus, $[\Delta_i^{up}]=[\mathcal{L}_i^{up}]$ and $[\Delta_i^{down}]=[\mathcal{L}_i^{down}]$.
\end{proof}

\section{Eigenvalues and Eigenfunctions of the Weighted Laplacian}
In this section we study the eigenvalues and eigenfunctions of the $\phi$-weighted Laplacian (Definition \ref{defn:weightedlap}) corresponding to the $\phi$-weighted coboundary map $\delta$.

\begin{defn}[{\cite[p.~140]{halmos2012finite}}]
A linear transformation $A$ on an inner product space is said to be non-negative (or positive semi-definite) if it is self-adjoint and if $(Ax,x)\geq 0$ for all $x$.
\end{defn}

\begin{theorem}[{\cite[p.~153]{halmos2012finite}}]
\label{thm:nonnegative}
If $A$ is a self-adjoint transformation on an inner product space, then every eigenvalue of $A$ is real. Furthermore, if $A$ is non-negative, then every eigenvalue of $A$ is non-negative.
\end{theorem}

\begin{cor}
The eigenvalues of the three $\phi$-weighted operators $\delta_{n-1}\delta_{n-1}^*$, $\delta_n^*\delta_n$ and $\Delta_n$ are all real and non-negative.
\end{cor}
\begin{proof}
We can verify from the definition of the three operators that they are self-adjoint and non-negative. Then, we apply Theorem \ref{thm:nonnegative} to conclude the proof.
\end{proof}

Next, we calculate the multiplicity of the eigenvalue zero in the spectrum of the three $\phi$-weighted operators $\delta_{n-1}\delta_{n-1}^*$, $\delta_n^*\delta_n$ and $\Delta_n$, where $\delta$ is the $\phi$-weighted coboundary map. This is the $\phi$-weighted version of Theorem 3.1 in \cite{Horak2013}. For a self-adjoint transformation on a finite-dimensional inner product space, the algebraic multiplicity of each eigenvalue is equal to its geometric multiplicity (cf.\ \cite[p.~154]{halmos2012finite}).

\begin{remark}
In the process of generalizing Theorem 3.1 in \cite{Horak2013}, we have spotted minor errata and have corrected them in our Theorem \ref{thm:multiplicity}. The summation $\sum_{j=0}^{i}$ in the statement of Theorem 3.1 \cite[p.~311]{Horak2013} should be starting from -1 instead of 0. This is because $\dim C^{-1}=\dim\R=1\neq 0$ to be consistent with their context of reduced cohomology. Also, $\dim\widetilde{H}^i$ should be replaced with $\dim C^i$ in the statement of Theorem 3.1 (ii) in \cite{Horak2013}.

It should be remarked that our Theorem \ref{thm:multiplicity} uses unreduced cohomology instead of reduced cohomology, hence both $\dim C^{-1}(K;\F)$ and $\dim{H}^{-1}(K,\phi;\F)$ are 0.
\end{remark}

\begin{theorem}[cf.\ {\cite[p.~311]{Horak2013}}]
\label{thm:multiplicity}
Let $\F=\R$ or $\C$. Let $\delta_n$ be the $n$-th $\phi$-weighted coboundary operator induced by the weight function $\phi:K\times K\to\F$. Let $\dim V:=\dim_\F V$ denote the dimension of an $\F$-vector space $V$.

The multiplicity of the eigenvalue zero for $\delta_{n-1}\delta_{n-1}^*$ is 
\begin{equation}
\label{eq:zeroeigen1}
\dim C^n(K;\F)-\sum_{j=0}^{n-1}(-1)^{n+j-1}(\dim C^j(K;\F)-\dim{H}^j(K,\phi;\F)).
\end{equation}

The multiplicity of the eigenvalue zero for $\delta_n^*\delta_n$ is 
\begin{equation}
\label{eq:zeroeigen2}
\dim C^n(K;\F)-\sum_{j=0}^n(-1)^{n+j}(\dim C^j(K;\F)-\dim{H}^j(K,\phi;\F)).
\end{equation}

The multiplicity of the eigenvalue zero for $\Delta_n$ is
\begin{equation}
\label{eq:zeroeigen3}
\dim{H}^n(K,\phi;\F).
\end{equation}
\end{theorem}
\begin{proof}
Note that $\Ima\delta_n$ and ${H}^n(K,\phi;\F)$ are vector spaces over $\F$, hence they are projective modules. Hence the following are split exact sequences:
\begin{align*}
0&\to\ker\delta_n\to C^n(K;\F)\to\Ima\delta_n\to 0,\\
0&\to\Ima\delta_{n-1}\to\ker\delta_n\to{H}^n(K,\phi;\F)\to 0.
\end{align*}

Thus we have
\begin{align}
\label{align:Cn}
\dim C^n(K;\F)&=\dim\ker\delta_n+\dim\Ima\delta_n,\\
\label{align:ker}
\dim\ker\delta_n&=\dim{H}^n(K,\phi;\F)+\dim\Ima\delta_{n-1}.
\end{align}

Substituting (\ref{align:ker}) into (\ref{align:Cn}), we have
\begin{equation*}
\dim C^n(K;\F)-\dim{H}^n(K,\phi;\F)=\dim\Ima\delta_{n-1}+\dim\Ima\delta_n
\end{equation*}
and hence
\begin{equation*}
\begin{split}
\dim\Ima\delta_n&=\sum_{j=-1}^n(-1)^{n+j}(\dim\Ima\delta_{j-1}+\dim\Ima\delta_j)\\
&=\sum_{j=0}^n (-1)^{n+j}(\dim C^j(K;\F)-\dim{H}^j(K,\phi;\F)).
\end{split}
\end{equation*}

Next, we note that the multiplicity of the eigenvalue zero for $\delta_{n-1}\delta_{n-1}^*$ is
\begin{equation*}
\begin{split}
\dim\ker\delta_{n-1}\delta_{n-1}^*&=\dim\ker\delta_{n-1}^*\\
&=\dim(\Ima\delta_{n-1})^\perp\\
&=\dim C^n(K;\F)-\dim\Ima\delta_{n-1}\\
&=\dim C^n(K;\F)-\sum_{j=0}^{n-1}(-1)^{n-1+j}(\dim C^j(K;\F)-\dim{H}^j(K,\phi;\F)).
\end{split}
\end{equation*}

We have proved (\ref{eq:zeroeigen1}).

To prove (\ref{eq:zeroeigen2}), we calculate that
\begin{equation*}
\begin{split}
\dim\ker\delta_n^*\delta_n&=\dim\ker\delta_n\\
&=\dim C^n(K;\F)-\sum_{j=0}^n (-1)^{n+j}(\dim C^j(K;\F)-\dim{H}^j(K,\phi;\F)).
\end{split}
\end{equation*}

To prove (\ref{eq:zeroeigen3}), we use Theorem \ref{thm:mainiso} to conclude that
\begin{equation*}
\dim\ker\Delta_n=\dim{H}^n(K,\phi;\F).
\end{equation*}
\end{proof}

\subsection{Harmonic Cochains (Eigenfunctions of the $\phi$-weighted Laplacian corresponding to the eigenvalue 0)}
In \cite{cappell2006cohomology}, the authors study the cohomology of harmonic forms on Riemannian manifolds with boundary. We apply the idea to study the harmonic cochains on simplicial complexes arising from the $\phi$-weighted Laplacian $\Delta$. Harmonic cochains are \emph{eigenfunctions} of the $\phi$-weighted Laplacian corresponding to the eigenvalue 0.

\begin{defn}
Let $\F=\R$ or $\C$. A cochain $f\in C^n(K;\F)$ is called \emph{harmonic} if $\Delta_n f=0$. In other words, a harmonic cochain $f$ is an \emph{eigenfunction} of $\Delta_n$ corresponding to the eigenvalue 0. We denote the subset of harmonic cochains in $C^n(K;\F)$ by $\Harm^n(K,\phi):=\ker\Delta_n$, where $\phi:K\times K\to\F$ is the weight function used in the definition of the $\phi$-weighted coboundary map.
\end{defn}

\begin{lemma}
\label{prop:commute}
The $\phi$-weighted coboundary map $\delta$ commutes with the $\phi$-weighted Laplacian $\Delta$. To be precise, we have $\Delta_{n+1}\delta_n=\delta_n\Delta_n$ for all $n\geq 0$.
\end{lemma}
\begin{proof}
We have
\begin{equation*}
\begin{split}
\Delta_{n+1}\delta_n&=(\delta_n\delta_n^*+\delta_{n+1}^*\delta_{n+1})(\delta_n)\\
&=\delta_n\delta_n^*\delta_n+0\\
&=\delta_n(\delta_{n-1}\delta_{n-1}^*+\delta_n^*\delta_n)\\
&=\delta_n\Delta_n.
\end{split}
\end{equation*}
\end{proof}

\begin{prop}
$(\Harm^*(K,\phi),\delta)$ is a sub-cochain complex of the cochain complex $(C^*(K;\F),\delta)$.

(In $(\Harm^*(K,\phi),\delta)$, the map $\delta$ is the restriction of the original $\phi$-weighted coboundary map of $C^*(K;\F)$.)
\end{prop}
\begin{proof}
Since $\delta$ is the dual of $\partial$, it follows that $\delta^2=0$. Hence it remains to show that $\delta$ preserves harmonicity of cochains, so that $\delta_n: \Harm^n(K,\phi)\to\Harm^{n+1}(K,\phi)$ is a coboundary map.

Let $f\in\Harm^n(K,\phi)$, i.e.\ $\Delta_n f=0$. By Lemma \ref{prop:commute}, we have 
\begin{equation*}
\Delta_{n+1}(\delta_n f)=\delta_n(\Delta_n f)=0.
\end{equation*}
Hence, we have shown that $\delta$ preserves harmonicity of cochains.
\end{proof}

It is natural to try to compute the cohomology of $(\Harm^*(K,\phi),\delta)$, which we call the \emph{harmonic cohomology} of $(K,\phi)$.

\begin{theorem}
We have the following isomorphism
\begin{equation*}
H^n(\Harm^*(K,\phi),\delta)=\Harm^n(K,\phi)\cong{H}^n(K,\phi;\F).
\end{equation*}
\end{theorem}
\begin{proof}
In the proof of Theorem \ref{thm:mainiso} (Equation \ref{eqn:perpcong}), we have shown that $\ker\Delta_n=\ker\delta_{n-1}^*\cap\ker\delta_n$. Hence for a harmonic cochain $f$ satisfying $\Delta_n f=0$, we have $\delta_n f=0$. This means that all the coboundary maps in $(\Harm^*(K,\phi),\delta)$ are zero, and hence
\begin{equation*}
\begin{split}
H^n(\Harm^*(K,\phi),\delta)&=\Harm^n(K,\phi)\\
&=\ker\Delta_n\\
&\cong{H}^n(K,\phi;\F).
\end{split}
\end{equation*}

The isomorphism in the last line is due to Theorem \ref{thm:mainiso}.
\end{proof}

\begin{remark}
We remark that this kind of result (that harmonic cohomology is isomorphic to the cohomology) holds in the case of closed manifolds. However, when a manifold $M$ is connected and has non-empty boundary, it is possible for the result not to hold true \cite{cappell2006cohomology}.
\end{remark}

\section{Applications}
We outline some possible applications of $\phi$-weighted (co)homology and the $\phi$-weighted Laplacian.

\subsection{Applications to Weighted Polygons}
\label{subsec:biomolecular}
In this subsection, we study the $\phi$-weighted homology of weighted polygons, which are special cases of 1-dimensional weighted simplicial complexes. A motivating example is that of ring structures, which are widely found in biomolecules (e.g.\ proteins \cite{xia2018multiscale,engh1991accurate,laskowski1993main}, DNA and RNA). The atoms that are in or attached to these ring structures are usually different, resulting in various ring sizes and different bond angles.

We show in Example \ref{eg:pentagon} that $\phi$-weighted homology is able to distinguish between pentagons with different \emph{angle weights}. The $\phi$-weighted homology could potentially be used to strengthen existing topological data analysis tools, such as persistent homology, which are effectively used in the data analysis of biomolecules \cite{xia2014persistent,xia2015multiresolution,xia2015persistent,cang2017topologynet,cang2018persistent,cang2018representability}.
\begin{figure}[htbp]
\begin{center}
\begin{tikzpicture}[scale=0.75]
\coordinate (v0) at (-2.2,1.5);
\coordinate (v1) at (0,0);
\coordinate (v2) at (2.2,1.5);
\coordinate (v3) at (1,3);
\coordinate (v4) at (-1,3);

\filldraw 
(v1) circle (2pt) node[align=left,   below] {$v_1$}
-- (v2) circle (2pt) node[align=center, below, right] {$v_2$}  
-- (v3) circle (2pt) node[align=left,   above] {$v_3$} 
-- (v4) circle (2pt) node[align=left, above] {$v_4$}
-- (v0) circle (2pt) node[align=left,   above, left] {$v_0$} 
-- (v1);
\pic [draw, -, "$\alpha_0$", angle eccentricity=1.5] {angle = v1--v0--v4};
\pic [draw, -, "$\alpha_1$", angle eccentricity=1.5] {angle = v2--v1--v0};
\pic [draw, -, "$\alpha_2$", angle eccentricity=1.5] {angle = v3--v2--v1};
\pic [draw, -, "$\alpha_3$", angle eccentricity=1.5] {angle = v4--v3--v2};
\pic [draw,-,"$\alpha_4$", angle eccentricity=1.5] {angle = v0--v4--v3};
\end{tikzpicture}
\caption{The simplicial complex $K$ with five vertices and five edges.}
\label{fig:pentagon}
\end{center}
\end{figure}
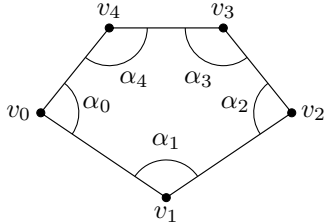

The Smith normal form is commonly used to compute integral homology \cite{dumas2003computing,kaczynski2004computational}. We will use the following fact about Smith normal form:
\begin{prop}[cf.\ {\cite[Prop.~8.1]{miller2009differential}}]
\label{prop:diagsmith}
If the diagonal entries in the Smith normal form of a matrix $A$ are $d_1,d_2,\dots$, then the partial product $d_1d_2\dots d_k$ is the greatest common divisor of all determinants of $k\times k$ minors of the original matrix $A$.
\end{prop}

\begin{eg}
\label{eg:pentagon}
Consider the pentagon shown in Figure \ref{fig:pentagon}, viewed as a simplicial complex $K$ with five vertices and five edges. The \emph{angle weights} $\alpha_0,\dots,\alpha_4$ are chosen to represent the interior angles of the polygon. (Depending on the situation, there may be different possible schemes of assigning values to $\alpha_i$.)

Consider the weight function $\phi:K\times K\to\Z$ such that $\phi([v_i,v_j],[v_i])=\alpha_i$ and $\phi([v_i,v_j],[v_j])=\alpha_j$ for all $i<j$, and adopting the convention (cf.\ Remark \ref{remark:notface}) that $\phi(\sigma,d_i\sigma)=0$ if $\sigma$ is a 0-simplex, such that the conditions of Proposition \ref{prop:integercond} are satisfied. We note that for non-zero angle weights, this weighting is of Dawson type: each edge has weight $C$ and each vertex $v_i$ has weight $C/\alpha_i$, where $C$ is some large integer that all the $\alpha_i$ divide.

The matrix representation of the boundary operator is:
\begin{equation*}
[\partial_1]=\mleft[
\begin{array}{c|ccccc}
 &[v_0,v_1] &[v_0,v_4] &[v_1,v_2] &[v_2,v_3] &[v_3,v_4]\\
\hline
{[}v_0] &-\alpha_0 &-\alpha_0 &0 &0 &0\\
{[}v_1] &\alpha_1 &0 &-\alpha_1 &0 &0\\
{[}v_2] &0 &0 &\alpha_2 &-\alpha_2 &0\\
{[}v_3] &0 &0 &0 &\alpha_3 &-\alpha_3\\
{[}v_4] &0 &\alpha_4 &0 &0 &\alpha_4
\end{array}
\mright].
\end{equation*}

By Proposition \ref{prop:diagsmith}, we have:
\begin{align*}
d_1&=\gcd(\alpha_0,\alpha_1,\alpha_2,\alpha_3,\alpha_4)\\
d_1d_2&=\gcd(\{\alpha_i\alpha_j\mid i<j\})=\gcd(\alpha_0\alpha_1,\alpha_0\alpha_2,\dots,\alpha_3\alpha_4)\\
d_1d_2d_3&=\gcd(\{\alpha_i\alpha_j\alpha_k\mid i<j<k\})\\
d_1d_2d_3d_4&=\gcd(\{\alpha_i\alpha_j\alpha_k\alpha_l\mid i<j<k<l\})\\
d_1d_2d_3d_4d_5&=\gcd(\det[\partial_1])=0.
\end{align*}
This also implies that $d_5=0$. The resulting 0-th $\phi$-weighted homology is then:
\begin{equation*}
H_0(K,\phi;\Z)=\Z/d_1\oplus\Z/d_2\oplus\Z/d_3\oplus\Z/d_4\oplus\Z.
\end{equation*}

We see that different values of $\alpha_i$ can lead to different $\phi$-weighted homology. In particular, when $\alpha_0=\dots=\alpha_4=1$, we have $H_0(K,\phi;\Z)=\Z$, which corresponds to the usual unweighted homology.

For example, suppose that $\alpha_0=1$, $\alpha_1=\alpha_2=\alpha_3=\alpha_4=2$. Then $d_1=1$, and $d_2=d_3=d_4=2$. The resulting 0-th $\phi$-weighted homology is then:
\begin{equation*}
H_0(K,\phi;\Z)=\Z/2\oplus\Z/2\oplus\Z/2\oplus\Z.
\end{equation*}
\end{eg}

We may generalize Example \ref{eg:pentagon} to the case of the $n$-sided polygon (or $n$-gon for short) with $n$ angle weights. We remark that ``angle weights'' is just a convenient nomenclature. Depending on the application, ``angle weights'' may be used to reflect some other quantities, not just angles.

We remark that the argument above does not take into account directions of edges. An alternative approach that does so utilizing a nonsymmetric operator with complex eigenvalues is presented by F.\ Bauer \cite{bauer2012normalized}.

\begin{theorem}
\label{thm:ngon}
Let $K$ be a $n$-gon, viewed as a simplicial complex with $n$ vertices $v_0,\dots,v_{n-1}$ and $n$ edges in the edge set $E=\{[v_i,v_{i+1}]\mid 0\leq i\leq n-1\}$. (For convenience, we define $v_n:=v_0$.) Let $\alpha_0,\dots,\alpha_{n-1}\in\Z$ represent the $n$ angle weights of the $n$-gon $K$.

Let the weight function $\phi:K\times K\to\Z$ be such that $\phi([v_i,v_j],[v_i])=\alpha_i$, $\phi([v_i,v_j],[v_j])=\alpha_j$ for all $i<j$.

Then
\begin{equation*}
H_0(K,\phi;\Z)=\Z/d_1\oplus\dots\oplus\Z/d_{n-1}\oplus\Z,
\end{equation*}
where
\begin{equation}
\label{eq:gcdds}
d_1d_2\dots d_k=\gcd(\{\alpha_{i_1}\alpha_{i_2}\dots\alpha_{i_k}\mid {i_1<i_2<\dots<i_k}\})
\end{equation}
for $1\leq k\leq n-1$ and $d_n=0$.
\end{theorem}
\begin{proof}
By Proposition \ref{prop:diagsmith}, it suffices to show that:
\begin{enumerate}
\item The determinants of $k\times k$ minors of the matrix of the $\phi$-weighted boundary map $[\partial_1]$ are either 0 or of the form $\pm\alpha_{i_1}\alpha_{i_2}\dots\alpha_{i_k}$, $i_1<i_2<\dots<i_k$, for $1\leq k\leq n-1$. Furthermore, the above-mentioned determinants include all $n\choose k$ possible choices of $\alpha_{i_1}\alpha_{i_2}\dots\alpha_{i_k}$.
\item $\det[\partial_1]=0$.
\end{enumerate}

To prove (1), suppose $A$ is a $k\times k$ minor of $[\partial_1]$, where $1\leq k\leq n-1$. Each column of $A$ contains at most 2 nonzero entries by definition of the $\phi$-weighted boundary matrix $[\partial_1]$. Since $k<n$, there must exist one column $C$ of $A$ with at most 1 nonzero entry. If $C$ is a zero column, then $\det(A)=0$. Suppose $C$ has a nonzero entry, say $\pm\alpha_i$ for some $i$. We do cofactor expansion along $C$. The minor $A'$ of $A$ obtained by deleting the row and column containing $\pm\alpha_i$ is also a minor of $[\partial_1]$, hence again there exists a column $C'$ of $A'$ with at most 1 nonzero entry. If $C'$ is a zero column, then $\det(A)=0$. Suppose $C'$ has a nonzero entry, say $\pm\alpha_j$. Note that by construction of $[\partial_1]$, we have that $\alpha_i\neq\alpha_j$ since they are on different rows. Continuing cofactor expansion, we see that $\det(A)=0$ or $\det(A)=\pm\alpha_{i_1}\alpha_{i_2}\dots\alpha_{i_k}$, where $1\leq k\leq n-1$.

To show that all $n\choose k$ possible choices of $\alpha_{i_1}\alpha_{i_2}\dots\alpha_{i_k}$ indeed occur, we show that there exists a $k\times k$ minor with determinant $\pm\alpha_{i_1}\alpha_{i_2}\dots\alpha_{i_k}$. Since $k<n$, there exists $0\leq j\leq n-1$ such that $j$ is not equal to any of $i_1,i_2,\dots,i_k$. Without loss of generality, we assume $[v_{i_1},v_j]$ is in the edge set $E$. (If not, we may relabel the vertices which amounts to swapping rows/columns which only affects the determinant of the minor by a sign, and hence does not change the $\gcd$ in Equation \ref{eq:gcdds}. Also note that if $j<i_1$, we may just switch the order of the vertices to $[v_j,v_{i_1}]$ which only changes the entries in the column by a sign.)

Consider the matrix
\begin{equation*}
B=\mleft[
\begin{array}{c|ccccc}
 &[v_{i_1},v_j] &[v_{i_2-1},v_{i_2}] &\dots &[v_{i_{k-1}-1},v_{i_{k-1}}] &[v_{i_k-1},v_{i_k}]\\
\hline
{[}v_{i_1}] &-\alpha_{i_1} &* &\dots &* &*\\
{[}v_{i_2}] &0 &\alpha_{i_2} &\dots &* &*\\
{\vdots} &\vdots &\vdots &\vdots &\vdots &\vdots\\
{[}v_{i_{k-1}}] &0 &0 &\dots &\alpha_{i_{k-1}} &*\\
{[}v_{i_k}] &0 &0 &\dots &0 &\alpha_{i_k}
\end{array}
\mright]
\end{equation*}
where $*$ entries could be either 0 or nonzero.

We see that $B$ is upper triangular with $\det(B)=\pm\alpha_{i_1}\alpha_{i_2}\dots\alpha_{i_k}$. Note that $B$ is a minor of $[\partial_1]$ possibly after interchanging of rows/columns which preserves the determinant up to a sign change. We have proven (1) which implies that Equation \ref{eq:gcdds} holds for $1\leq k\leq n-1$.

To prove (2), observe that the columns of
\begin{equation*}
[\partial_1]=\mleft[
\begin{array}{c|cccc}
&[v_0,v_1] &[v_1,v_2] &\dots &[v_{n-1},v_0]\\
\hline
{[}v_0] &-\alpha_0 & 0 &\dots &\alpha_0\\
{[}v_1] &\alpha_1 &-\alpha_1 &\dots &0\\
\vdots &\vdots &\vdots &\vdots &\vdots\\
{[}v_{n-1}] &0 &0 &\dots &-\alpha_{n-1}
\end{array}
\mright]
\end{equation*}
sum up to zero when the column labels are $[v_i,v_{i+1}]$. Hence the columns of $[\partial_1]$ are linearly dependent which implies that $\det[\partial_1]=0$. (Note that rearrangement of columns/rows do not change the zero determinant since such operations amount to multiplying the zero determinant by -1.)
\end{proof}

\subsection{Applications to Weighted Laplacians of Digraphs}
\label{subsec:network}
A digraph can be viewed as a graph where each edge has a weight corresponding to its direction. That is, there is a total of 2 possible weights corresponding to the 2 different directions. We study the $\phi$-weighted Laplacians of certain digraphs, with the motivating example being network motifs. We remark that our methods can be used for essentially any digraph that may appear in other disciplines such as computer science. The subject of network motifs has found important applications in biology \cite{alon2007network,ferrell2013feedback,milo2002network,masoudi2012building}. In the review paper by Uri Alon \cite{alon2007network}, it is mentioned that recent work \cite{shen2002network,milo2002network} indicates that transcription networks contain a small set of network motifs.

An important family of network motifs is the feedforward loop (FFL) \cite[p.~452]{alon2007network}. This motif consists of three genes X,Y,Z. Gene X regulates genes Y and Z, while gene Y regulates gene Z. Each regulatory interaction in the FFL can be either activation or repression, leading to $2^3=8$ possible types of FFLs (cf.\ \cite[p.~452]{alon2007network}). In \cite[p.~452]{alon2007network}, activation is represented by the usual arrow (\begin{tikzpicture}[baseline=-0.5ex]\draw[-latex] (0,0)--(0.5,0);\end{tikzpicture}), while repression is represented by an arrow with a bar (\begin{tikzpicture}\draw[-|] (0,0)--(0.5,0);\end{tikzpicture}). We see that mathematically, each FFL is essentially a digraph.

We show that the eigenvalues and eigenvectors of the $\phi$-weighted Laplacian can distinguish between the 8 different types of FFLs. Firstly, we view the genes X,Y,Z as vertices of a simplicial complex $K$, and the regulatory interactions as edges. Hence, $K=\{[X],[Y],[Z],[X,Y],[Y,Z],[X,Z]\}$. We define the weight function $\phi:K\times K\to\R$ such that the matrix of the $\phi$-weighted boundary map is
\begin{equation*}
[\partial_1]=\mleft[
\begin{array}{c|ccc}
&[X,Y] &[Y,Z] &[X,Z]\\
\hline
{[}X] &-a & 0 &-c\\
{[}Y] &a &-b &0\\
{[}Z] &0 &b &c
\end{array}
\mright],
\end{equation*}
where $a,b,c\in\R$.

We calculate that the matrix of the $\phi$-weighted Laplacian is
\begin{equation*}
[\Delta_0]=\begin{bmatrix}
a^2+c^2 &-a^2 &-c^2\\
-a^2 & a^2+b^2 &-b^2\\
-c^2 &-b^2 &b^2+c^2
\end{bmatrix}.
\end{equation*}

Let $d=\sqrt{a^4+b^4+c^4-a^2b^2-a^2c^2-b^2c^2}$. For $b\neq c$, the eigenvectors of $[\Delta_0]$ are
\begin{equation*}
\mathbf{u_1}=\begin{pmatrix}1\\1\\1\end{pmatrix},\quad\mathbf{u_2}=\begin{pmatrix}-(d-a^2+b^2)/(b^2-c^2)\\(d-a^2+c^2)/(b^2-c^2)\\1\end{pmatrix}, \quad\mathbf{u_3}=\begin{pmatrix}(d+a^2-b^2)/(b^2-c^2)\\-(d+a^2-c^2)/(b^2-c^2)\\1\end{pmatrix}
\end{equation*}

corresponding to eigenvalues $\lambda_1=0$, $\lambda_2=-d+a^2+b^2+c^2$, and $\lambda_3=d+a^2+b^2+c^2$ respectively.

For $b=c$, the eigenvectors of $[\Delta_0]$ are $\mathbf{u_1}=\begin{pmatrix}1\\1\\1\end{pmatrix}$, 
$\mathbf{u_2}=\begin{pmatrix}-1/2\\-1/2\\1\end{pmatrix}$, $\mathbf{u_3}=\begin{pmatrix}-1\\1\\0\end{pmatrix}$ corresponding to eigenvalues $\lambda_1=0$, $\lambda_2=3b^2$, $\lambda_3=2a^2+b^2$ respectively.

By choosing suitable values of $a,b,c$, we can distinguish between the 8 types of FFLs. For instance, we can choose $a=1$ if the edge $[X,Y]$ represents an activation interaction while we choose $a=2$ if $[X,Y]$ represents a repression interaction. Similarly, we choose $b,c=1$ or $2$ depending on whether $[Y,Z], [X,Z]$ are activation or repression interactions respectively. It can be verified that such choices of $a,b,c$ can lead to eigenvectors and eigenvalues that distinguish between all 8 types of FFLs, as shown in Table \ref{table:abc}.

\begin{table}[htbp!]
\centering
\caption{Table showing eigenvectors/eigenvalues of the $\phi$-weighted Laplacian.}
 \label{table:abc}
 \begin{tabular}{||c||c|c|c||c|c|c|c||} 
 \hline
 Type of FFL & $a$ & $b$ & $c$ & $\mathbf{u_2}$ & $\mathbf{u_3}$ &$\lambda_2$ &$\lambda_3$\\ [0.5ex] 
 \hline\hline
 Coherent type 1 &1 &1 &1 &$(-0.5,-0.5,1)$ &$(-1,1,0)$ &3 &3\\
 \hline
 Coherent type 2 &2 &1 &2 &$(0,-1,1)$ &$(-2,1,1)$ &6 &12\\
 \hline
 Coherent type 3 &1 &2 &2 &$(-0.5,-0.5,1)$ &$(-1,1,0)$ &12 &6\\
 \hline
 Coherent type 4 &2 &2 &1 &$(-1,0,1)$ &$(1,-2,1)$ &6 &12\\
 \hline
 Incoherent type 1 &1 &2 &1 &$(-2,1,1)$ &$(0,-1,1)$ &3 &9\\
 \hline
 Incoherent type 2 &2 &2 &2 &$(-0.5,-0.5,1)$ &$(-1,1,0)$ &12 &12\\
 \hline
 Incoherent type 3 &1 &1 &2 &$(1,-2,1)$ &$(-1,0,1)$ &3 &9\\
 \hline
 Incoherent type 4 &2 &1 &1 &$(-0.5,-0.5,1)$ &$(-1,1,0)$ &3 &9\\
 \hline
 \end{tabular}
 \end{table}

\begin{remark}
We remark that an advantage of our definition is that the $\phi$-weighted Laplacian for digraphs generalizes easily to higher dimensions (directed simplicial complexes), and also to cases where there are more than two types of arrows (e.g.\ weighted graphs/digraphs). This is because firstly, our $\phi$-weighted Laplacian is originally defined for dimension $n$ for all $n\geq 0$. Secondly, in the case of more than two types of arrows, we have the ability to choose a different weight for each type of arrow.
\end{remark}

\section{Future Directions}
In Example \ref{eg:pentagon}, sometimes it may not be easy to find the bond angles, as a node can be shared by more than two edges. In this case, other node or edge information such as degree, centrality, cluster coefficient, discrete Ricci curvatures could be studied further as these are geometric properties that are well-defined on graphs. 

We remark that a drawback of the bond angle application, in the case where the weight function $\phi$ takes values in $\mathbb{Z}$, is that the angles have to be integers. In practice, this could be remedied by rounding off the angles to the nearest integer, or alternatively multiplying all angles by a sufficiently large power of 10.

In the second application (Subsection \ref{subsec:network}), it will be interesting to find a suitable weight function that is able to discriminate the type of feedforward loops using purely homology information. Finally, persistent homology has proven to be a very important tool in analyzing the shape of data, it will be useful to extend the current work to persistent homology.

\section*{Acknowledgements}
We wish to thank most warmly the referees of Houston Journal of Mathematics for numerous suggestions that have improved the exposition of this paper.


\bibliographystyle{amsplain}
\bibliography{jabref9}

\end{document}